\documentclass[11pt]{article}

\usepackage[T1]{fontenc}
\usepackage[utf8]{inputenc}
\usepackage{lmodern}
\usepackage[a4paper,margin=1in]{geometry}
\usepackage{amsmath,amssymb,amsthm,mathtools}
\usepackage{microtype}
\usepackage{xcolor}
\usepackage[colorlinks=true,
            linkcolor=blue!45!black,
            citecolor=blue!45!black,
            urlcolor=blue!45!black]{hyperref}

\allowdisplaybreaks

\newtheorem{theorem}{Theorem}[section]
\newtheorem{lemma}[theorem]{Lemma}
\newtheorem{proposition}[theorem]{Proposition}
\newtheorem{corollary}[theorem]{Corollary}
\theoremstyle{definition}
\newtheorem{problem}[theorem]{Problem}

\newcommand{\cpn}{\operatorname{cp}}
\newcommand{\ccn}{\operatorname{cc}}

\title{On the difference between clique partition and clique covering numbers}
\author{Bo Ning}
\date{30 July 2026}

\begin{document}

\maketitle

\begin{abstract}
For a graph $G$, let $\cpn(G)$ and $\ccn(G)$ denote the minimum numbers of
cliques whose edge sets partition and cover $E(G)$, respectively, and put
$f(n)=\max_{|V(G)|=n}\bigl(\cpn(G)-\ccn(G)\bigr).$
In 1983, Erd\H{o}s, Faudree, and Ordman asked whether there is a sequence of
graphs $G_n$ such that $|V(G_n)|=n$ and
$\cpn(G_n)-\ccn(G_n)=n^2/4+O(n)$. The question appears as Problem 66 in Chung’s survey \cite{ChungProblems} and is also listed on the
UCSD Erd\H{o}s Problems website. Caccetta, Erd\H{o}s, Ordman, and Pullman
proved that $f(n)=n^2/4-o(n^2)$. We prove that
$f(n)=\left\lfloor\frac{n^2}{4}\right\rfloor-\Theta(n^{4/3}),$
and hence answer the question in the negative.
\end{abstract}

\section{Introduction}

In this paper, we consider graphs which are finite, simple, and undirected. A  \emph{clique
cover} of a graph $G$ is a family $\mathcal C$ of cliques such that
$E(G)\subseteq\bigcup_{Q\in\mathcal C}E(Q)$. If the sets $E(Q)$,
$Q\in\mathcal C$, partition $E(G)$, then $\mathcal C$ is a \emph{clique
partition}. That is, a \emph{clique partition} of \(G\) is a family of
cliques $\mathcal{C}=\{C_1,C_2,\ldots,C_t\}$
such that
$E(G)=E(C_1)\,\dot\cup\,E(C_2)\,\dot\cup\cdots\dot\cup\,E(C_t),$
where \(\dot\cup\) denotes disjoint union. We denote the minimum cardinalities of a clique cover and a clique
partition of $G$ by $\ccn(G)$ and $\cpn(G)$, respectively. Thus,
$\ccn(G)\leq \cpn(G)$.

Clique covers arise naturally from intersection representations. Given a
clique cover of $G$, assign to each vertex the set of covering cliques that
contain it. Two vertices are adjacent precisely when their assigned sets
intersect. Conversely, every intersection representation of $G$ gives a
clique cover. Requiring two representing sets to intersect in at most one
element corresponds to a clique partition and connects the problem with finite
linear spaces; see de Bruijn and Erd\H{o}s~\cite{deBruijnErdos}.

For clique partition, a classic theorem due to de Bruijn and Erd\H{o}s~\cite{deBruijnErdos} states that for $n\geq 3$, if $\mathcal C$ is a non-trivial clique partition of $K_n$, then $|\mathcal C|\geq n$.
Erd\H{o}s, Goodman, and P\'osa~\cite{EGP} proved that every $n$-vertex graph
$G$ satisfies
\begin{equation}
 \cpn(G)\leq \left\lfloor\frac{n^2}{4}\right\rfloor,
 \label{eq:EGP}
\end{equation}
and that the partition may be chosen to consist only of edges and triangles.
The balanced complete bipartite graph shows that \eqref{eq:EGP} is sharp. For
this graph, however, every clique containing an edge has order two, so
$\ccn(G)=\cpn(G)$. A
weighted strengthening of \eqref{eq:EGP}, proved independently by Gy\H{o}ri
and Kostochka~\cite{GyoriKostochka}, Chung~\cite{ChungDecomposition}, and
Kahn~\cite{Kahn}, will be used in the proof of our upper bound.

Set
\[
 f(n)=\max_{|V(G)|=n}\bigl(\cpn(G)-\ccn(G)\bigr)
 \quad\text{and}\quad
 g(n)=\left\lfloor\frac{n^2}{4}\right\rfloor-f(n).
\]
In a personal communication in 1982, Erd\H{o}s asked for the order of the
largest possible difference. Caccetta, Erd\H{o}s, Ordman, and
Pullman~\cite{CEOP} determined its first-order behaviour. They constructed
$n$-vertex graphs $G_n$ satisfying
\begin{equation}
 \cpn(G_n)-\ccn(G_n)
   =\frac{n^2}{4}-\frac{n^{3/2}}{2}+\frac n4+O(1),
 \label{eq:CEOP-construction}
\end{equation}
and proved that
$f(n)\leq \lfloor n^2/4\rfloor-2$. Consequently,
$f(n)\sim n^2/4$, while \eqref{eq:CEOP-construction} gives
$g(n)=O(n^{3/2})$.

The Erd\H{o}s Problems website~\cite{ErdosProblemsWebsite} attributes the
following question to Erd\H{o}s, Faudree, and Ordman in 1983.

\begin{problem}[Erd\H{o}s--Faudree--Ordman]\label{Prob:1.1}
Is there a sequence of graphs $G_n$ with $|V(G_n)|=n$ such that
\[
 \cpn(G_n)-\ccn(G_n)=\frac{n^2}{4}+O(n)?
\]
\end{problem}

Chung recorded this question, together with the corresponding ratio problem,
as Problem~66 in~\cite{ChungProblems}. Faudree, Rousseau, and
Schelp~\cite[Questions~6.31--6.32]{FRS} stated two questions separately:
Question~6.31 concerns the largest possible ratio $\cpn(G)/\ccn(G)$, whereas
Question~6.32 asks whether $g(n)=O(n)$.

In this paper, we prove that $g(n)$ has order $n^{4/3}$.

\begin{theorem}\label{Thm:Main}
There are absolute constants $C_1,C_2>0$ such that
\[
 \left\lfloor\frac{n^2}{4}\right\rfloor-C_1n^{4/3}
 \leq f(n)\leq
 \left\lfloor\frac{n^2}{4}\right\rfloor-C_2n^{4/3}
\]
for every sufficiently large integer $n$.
\end{theorem}

Thus, $g(n)=\Theta(n^{4/3})$, and Problem~\ref{Prob:1.1} has a negative answer.

We briefly mention related work. Erd\H{o}s, Faudree, and
Ordman~\cite{EFO} studied several other measures of the difference between
clique covers and clique partitions. If $r=cn^a$, where $1/2<a<1$, they proved
that $\cpn(K_n-K_r)\sim c^2n^{2a}$. They also constructed graphs for which
$\cpn(G)/\ccn(G)>n^2/64$, and showed that the covering number of $K_n$ with a
matching removed lies between $\log n-1$ and $2\log n$. Their crossing-edge
inequality will be used to determine the partition number of the construction
in Section~\ref{sec:preliminaries}.

Erd\H{o}s, Ordman, and Zalcstein~\cite{EOZThreshold} considered threshold
coverings and partitions. They bounded the corresponding maximum parameters
between $n-A\sqrt{n\log n}$ and $n-\sqrt n+1$, for an absolute constant $A$,
and constructed planar graphs whose minimum threshold partition has $3/2$
times as many members as a minimum threshold cover. They later studied clique
partitions of chordal graphs~\cite{EOZChordal}, constructing chordal graphs
that require $n^2/6$ cliques and proving an upper bound $(1-c)n^2/4$ for some
$c>0$. For a survey of
graph covering and partitioning problems, see Schwartz~\cite{Schwartz}. Earlier
work on the two clique parameters includes Orlin~\cite{Orlin}, Pullman and
Donald~\cite{PullmanDonald}, and Wallis~\cite{Wallis}.

\medskip

\noindent
{\bf{Outline of the proof.}}
For the lower bound, we consider
$G_{h,k}=(hK_k)\vee(hK_k)$. A one-factorization of $K_k$ and the
crossing-edge inequality determine both clique parameters of $G_{h,k}$.
Choosing $hk$ of order $n$ and balancing the two error terms gives a deficit
of order $n^{4/3}$.

For the upper bound, a minimum-weight clique partition yields a dense
triangle-free spanning subgraph and hence a nearly balanced partition
$A\mathbin{\dot\cup}B$ of the vertex set. If $D$ denotes the deficit and
$L=e(G[A])+e(G[B])$, the covering number gives
$L=\Omega(n^2/\sqrt D)$, while a triangle-repacking argument gives $L=O(D)$.
It follows that $D=\Omega(n^{4/3})$. The proof uses the weighted
Erd\H{o}s--Goodman--P\'osa theorem, the Caro--Wei bound, Vizing's theorem, and
the crossing-edge inequality.

\medskip

\noindent
{\bf Organization.}
In Section~\ref{sec:preliminaries}, we introduce notation and
state the results used later.  In Section~\ref{sec:lower-construction}, we give the lower bound construction and compute its two clique parameters.
In Section~\ref{sec:upper}, we  prove the upper bound and complete the proof
of Theorem~\ref{Thm:Main}.

\section{Preliminaries}\label{sec:preliminaries}

For a graph $G$, let $e(G)=|E(G)|$, let $d_G(v)$ and $N_G(v)$ denote the degree
and open neighborhood of a vertex $v$, and let $\Delta(G)$ and $\alpha(G)$
denote the maximum degree and independence number of $G$, respectively. For
disjoint sets $A,B\subseteq V(G)$, let $e_G(A,B)$ be the number of edges with
one end in $A$ and the other in $B$, and put
\[
 \overline e_G(A,B)=|A||B|-e_G(A,B).
\]
Thus, $\overline e_G(A,B)$ is the number of missing edges between $A$ and $B$.
A clique in a clique cover or partition is allowed to have order two. For a
clique partition $\mathcal Q$, define
\[
 w(\mathcal Q)=\sum_{Q\in\mathcal Q}|Q|.
\]

We use the following weighted form of the Erd\H{o}s--Goodman--P\'osa theorem.

\begin{theorem}[Gy\H{o}ri--Kostochka~\cite{GyoriKostochka},
Chung~\cite{ChungDecomposition}, and Kahn~\cite{Kahn}]
Every graph $G$ on $n$ vertices has a clique partition $\mathcal Q$ such that
\[
 w(\mathcal Q)\leq 2\left\lfloor\frac{n^2}{4}\right\rfloor.
\]
\end{theorem}

We shall also use the Caro--Wei bound.

\begin{theorem}[Caro~\cite{Caro} and Wei~\cite{Wei}]
If $G$ has $n$ vertices and $m$ edges, then
\[
 \alpha(G)\geq
 \sum_{v\in V(G)}\frac{1}{d_G(v)+1}
 \geq \frac{n^2}{2m+n}.
\]
Consequently, if $\alpha(G)\leq x$ then
\begin{equation}
 e(G)\geq \frac{n^2}{2x}-\frac n2.
 \label{eq:Caro-Wei-consequence}
\end{equation}
\end{theorem}

We use Vizing's
edge-coloring theorem in Section~\ref{sec:upper}.

\begin{theorem}[Vizing~\cite{Vizing}]
Every graph $G$ has a proper edge-coloring with at most $\Delta(G)+1$ colors.
\end{theorem}

Finally, we use the following inequality of Erd\H{o}s, Faudree, and
Ordman~(see \cite[Lemma~4]{EFO}).

\begin{lemma}[Erd\H{o}s--Faudree--Ordman]\label{Lem:EFO}
Let $G$ be a graph with two parts $A$ and $B$ such that $V(G)=A\,\dot\cup\,B$. Suppose that $e(A,B)=s$, $e(A)=a$ and $e(B)=b$. Then
\[
 \cpn(G)\geq s-a-b-\min\{a,b\}.
\]
\end{lemma}

\section{The lower construction}
\label{sec:lower-construction}

For positive integers $h$ and $k$, let $hK_k$ denote the disjoint union of
$h$ copies of $K_k$. The join $G_1\vee G_2$ is obtained from disjoint copies of
$G_1$ and $G_2$ by adding all edges between them. Define
\[
 G_{h,k}=(hK_k)\vee(hK_k).
\]
Write
\[
 A=A_0\mathbin{\dot\cup}\cdots\mathbin{\dot\cup}A_{h-1},
 \qquad
 B=B_0\mathbin{\dot\cup}\cdots\mathbin{\dot\cup}B_{h-1},
\]
where each $A_i$ and each $B_j$ induces a copy of $K_k$. There are no edges
between distinct parts on the same side, and every vertex of $A$ is
adjacent to every vertex of $B$. In the following, we say each $A_i$ or $B_j$ is a part in $A\cup B$.

For each even $k$, both clique parameters of $G_{h,k}$ can be determined exactly.

\begin{proposition}
If $k$ is even and $h\geq k-1$, then
\[
 \ccn(G_{h,k})=h^2
 \quad\text{and}\quad
 \cpn(G_{h,k})=h^2k^2-\frac32hk(k-1).
\]
\end{proposition}

\begin{proof}
The $h^2$ sets $A_i\cup B_j$ induce cliques (and are all maximal cliques in $G_{h,k}$). Hence, the sets form a clique cover, so
$\ccn(G_{h,k})\leq h^2$. For the reverse inequality, choose $a_i\in A_i$ and
$b_j\in B_j$ for $0\leq i,j<h$. These vertices $\{a_i,b_j:0\leq i,j<h\}$ induce $K_{h,h}$. Every clique
of $G_{h,k}$ contains at most one selected vertex in $A$ and at most one
selected vertex in $B$, as any vertex $a_i\in A_i$ is non-adjacent to any vertex $a_j\in A_j$ where $0\leq i,j\leq h-1$ and $i\neq j$ and so does $B_i$ and $B_j$. Therefore, each clique covers at most one edge of this
$K_{h,h}$. Hence $\ccn(G_{h,k})\geq h^2$.

Next, we determine the partition number of $G_{h,k}$. 
There are $h^2k^2$ edges in $e(A,B)$ and $h\binom{k}{2}$ edges within each side $A$ and $B$. Applying Lemma~\ref{Lem:EFO} to the cut $(A,B)$,
we have
\begin{equation}
 \cpn(G_{h,k})
 \geq h^2k^2-3h\binom{k}{2}
 =h^2k^2-\frac32hk(k-1).
 \label{eq:construction-lower}
\end{equation}

It remains to construct a partition attaining equality in
\eqref{eq:construction-lower}. Since $k$ is even, $K_k$ has a
one-factorization~(see \cite[Section~3.3]{Stinson}). Label the $k-1$ one-factors by
$0,\ldots,k-2$, using the same labeling in every $A_i$ and every $B_j$.
The indices of parts $A_i$ and $B_j$ are taken modulo $h$.

For each $0\leq i<h$ and $0\leq r\leq k-2$, pair the $k/2$ edges in the $r$-th 1-factor
of $A_i$ bijectively with the $k/2$ edges in the $r$-th 1-factor of $B_{i+r}$.
For each paired pair of edges, place in the partition the $K_4$ induced by
their four ends. Since $h\geq k-1$, the part pairs $(A_i,B_{i+r})$ are
distinct as $r$ varies.

Every edge within $A$ now lies in exactly one of these copies of $K_4$. The
same is true within $B$: an edge in factor $r$ of $B_j$ is paired with an edge
of $A_{j-r}$. No crossing edge is repeated. Indeed, within a fixed part
pair the factor edges are matchings, so the corresponding copies of $K_4$ are
vertex-disjoint, and distinct pairs $(i,r)$ correspond to distinct ordered
part pairs.

There are
\[
 h\binom{k}{2}=\frac12hk(k-1)
\]
copies of $K_4$, and together they contain
\[
 4h\binom{k}{2}=2hk(k-1)
\]
crossing edges from $E(A,B)$. Use every remaining crossing edge as a two-vertex clique. The
resulting clique partition has
\[
 h\binom{k}{2}
 +\left(h^2k^2-4h\binom{k}{2}\right)
 =h^2k^2-3h\binom{k}{2}
 =h^2k^2-\frac32hk(k-1)
\]
members. Together with \eqref{eq:construction-lower}, this proves the result.
\end{proof}

Balancing the two error terms gives the lower bound in Theorem~1.2.

\begin{corollary}
There is an absolute constant $C>0$ such that
\[
 f(n)\geq \left\lfloor\frac{n^2}{4}\right\rfloor-Cn^{4/3}
\]
for every sufficiently large integer $n$.
\end{corollary}

\begin{proof}
Let $k$ be the least even integer with $k\geq n^{1/3}$, and put
$h=\lfloor n/(2k)\rfloor$. Then, we have
\[
 n^{1/3}\leq k<n^{1/3}+2,
\]
and $h\geq k-1$ for all sufficiently large $n$. Add $n-2hk$ isolated vertices
to $G_{h,k}$; this changes neither clique parameter. Proposition~3.1 gives
\[
 \cpn(G_{h,k})-\ccn(G_{h,k})
 =h^2k^2-\frac32hk(k-1)-h^2
 =\frac{(2hk)^2}{4}-O(hk^2+h^2).
\]
Now $hk=O(n)$, $k=O(n^{1/3})$, and $h=O(n^{2/3})$, so
$hk^2+h^2=O(n^{4/3})$. Moreover, $0\leq n-2hk<2k$, and hence
\[
 0\leq \frac{n^2-(2hk)^2}{4}
 =\frac{(n-2hk)(n+2hk)}{4}
 =O(nk)=O(n^{4/3}).
\]
The result follows after adjusting the absolute constant to account for the
floor.
\end{proof}

\section{The upper bound}
\label{sec:upper}

We first consider graphs of even order. Throughout this section, let $G$ be a
graph of even order $n$, and write
\[
 p=\cpn(G),\qquad t=\ccn(G),\qquad
 \delta=\frac{n^2}{4}-p,\qquad D=\delta+t.
\]
By \eqref{eq:EGP}, $\delta\geq0$, and
\[
 \cpn(G)-\ccn(G)=\frac{n^2}{4}-D.
\]

\begin{proposition}
There is an absolute constant $c_0>0$ such that every graph $G$ of even order
$n$ satisfies
\[
 \cpn(G)-\ccn(G)\leq \frac{n^2}{4}-c_0n^{4/3}
\]
whenever $n$ is sufficiently large.
\end{proposition}

The first lemma extracts a dense triangle-free graph from a minimum-weight
clique partition.

\begin{lemma}\label{Lem:4.2}
The graph $G$ has a triangle-free spanning subgraph $H$ such that
\[
 e(H)\geq \frac{n^2}{4}-3\delta.
\]
\end{lemma}

\begin{proof}
Choose a clique partition $\mathcal Q$ of minimum weight. We may assume that
$\mathcal Q$ has no one-vertex member. By Theorem~2.1,
$w(\mathcal Q)\leq n^2/2$. Let $s=|\mathcal Q|$, and let $r$ be the number of
members of $\mathcal Q$ having order at least three. Since
$s\geq p=n^2/4-\delta$,
\[
 \frac{n^2}{2}\geq w(\mathcal Q)\geq 2s+r
 \geq 2\left(\frac{n^2}{4}-\delta\right)+r.
\]
Therefore
\begin{equation}
 r\leq2\delta.
 \label{eq:number-large-cliques}
\end{equation}

Let $H$ consist of the edges that occur as two-vertex members of $\mathcal Q$.
Since $\mathcal Q$ partitions $E(G)$,
\[
 e(H)=s-r\geq \frac{n^2}{4}-\delta-2\delta
 =\frac{n^2}{4}-3\delta.
\]
If $H$ contained a triangle, the three corresponding two-vertex members of
$\mathcal Q$ could be replaced by that triangle, decreasing their total
weight from $6$ to $3$. This contradicts the choice of $\mathcal Q$.
\end{proof}

We need the following stability estimate of Mantel's theorem.

\begin{lemma}
Let $H$ be a triangle-free graph on $n$ vertices, and let
$e(H)=n^2/4-\Delta$. There is a partition
$V(H)=A\mathbin{\dot\cup}B$ such that $A$ is independent in $H$ and
\[
 e(H[B])\leq\Delta,\qquad
 \overline e_H(A,B)\leq2\Delta,
 \qquad
 \left||A|-\frac n2\right|\leq\sqrt{2\Delta}.
\]
\end{lemma}

\begin{proof}
Mantel's theorem~\cite{Mantel} gives $\Delta\geq0$. For $v\in V(H)$, put
$B_v=V(H)\setminus N_H(v)$. If $xy\in E(H)$, then
$N_H(x)\cap N_H(y)=\varnothing$. Hence
\[
 \sum_{v\in V(H)}e(H[B_v])
 =\sum_{xy\in E(H)}\bigl(n-d_H(x)-d_H(y)\bigr)
 =ne(H)-\sum_{v\in V(H)}d_H(v)^2.
\]
By the Cauchy--Schwarz inequality,
\[
 \sum_{v\in V(H)}d_H(v)^2
 \geq\frac{(2e(H))^2}{n}
 =\frac{4e(H)^2}{n}.
\]
It follows that
\[
 \sum_{v\in V(H)}e(H[B_v])
 \leq\frac{4e(H)(n^2/4-e(H))}{n}
 =\frac{4e(H)\Delta}{n}
 \leq n\Delta.
\]
Choose $v$ such that $e(H[B_v])\leq\Delta$, and set
$A=N_H(v)$ and $B=B_v$. Then $A$ is independent. Put
$a=|A|$, $b=e(H[B])$, and $\mu_H=\overline e_H(A,B)$. Since every edge of $H$
either crosses the cut or lies in $B$,
\[
 \mu_H=a(n-a)-(e(H)-b)
 =\Delta+b-\left(a-\frac n2\right)^2.
\]
As $0\leq b\leq\Delta$ and $\mu_H\geq0$, we obtain
\[
 \left(a-\frac n2\right)^2\leq\Delta+b\leq2\Delta
 \quad\text{and}\quad
 \mu_H\leq\Delta+b\leq2\Delta.
\]
\end{proof}

The same cut also controls the edges of $G$.

\begin{lemma}
There is a partition $V(G)=A\mathbin{\dot\cup}B$ such that, with
$\mu=\overline e_G(A,B)$,
\[
 \mu\leq6\delta,\qquad
 \left||A|-\frac n2\right|\leq\sqrt{6\delta},
\]
and each of $G[A]$ and $G[B]$ has a clique partition with at most $5\delta$
members.
\end{lemma}

\begin{proof}
Take the graph $H$ and the minimum-weight partition $\mathcal Q$ from the
proof of Lemma~\ref{Lem:4.2}, and put
\[
 \Delta=\frac{n^2}{4}-e(H)\leq3\delta.
\]
Apply Lemma~4.3 to $H$. Since $H\subseteq G$, every crossing edge missing from
$G$ is also missing from $H$. Therefore
\[
 \mu\leq\overline e_H(A,B)\leq2\Delta\leq6\delta,
\]
and
\[
 \left||A|-\frac n2\right|
 \leq\sqrt{2\Delta}\leq\sqrt{6\delta}.
\]

Intersect each member of $\mathcal Q$ with $A$, and discard intersections of
order at most one. The resulting cliques partition $E(G[A])$. Since $A$ is
independent in $H$, no two-vertex member of $\mathcal Q$ is retained. By
\eqref{eq:number-large-cliques}, the resulting partition has at most
$r\leq2\delta$ members.

The analogous restriction to $B$ may retain the $r$ members of order at least
three and the two-vertex members corresponding to $E(H[B])$. Thus, $G[B]$ has a
clique partition with at most
\[
 r+e(H[B])\leq2\delta+\Delta\leq5\delta
\]
members.
\end{proof}

The covering number now gives a lower bound on the number of edges within the
two parts.

\begin{lemma}
For the partition in Lemma~4.4, put
\[
 x=\alpha(G[A]),\qquad y=\alpha(G[B]),\qquad
 L=e(G[A])+e(G[B]).
\]
Then
\[
 xy\leq t+\mu\leq7D.
\]
If $D\leq n^2/48$, then
\[
 L\geq\frac{n^2/4-6D}{\sqrt{7D}}-\frac n2
 \geq\frac{n^2}{8\sqrt{7D}}-\frac n2.
\]
\end{lemma}

\begin{proof}
Choose independent sets $I\subseteq A$ and $J\subseteq B$ with
$|I|=x$ and $|J|=y$. A clique of $G$ contains at most one vertex of $I$ and at
most one vertex of $J$, so it covers at most one edge between $I$ and $J$. At
least $xy-\mu$ of the pairs between $I$ and $J$ are edges. Hence
\[
 t\geq xy-\mu,
\]
and Lemma~4.4 gives
\begin{equation}
 xy\leq t+\mu\leq t+6\delta\leq7D.
 \label{eq:xy-bound}
\end{equation}

Write $a=|A|$, so $|B|=n-a$. If $D\leq n^2/48$, Lemma~4.4 implies that both
$A$ and $B$ are nonempty. Thus, $x,y\geq1$, and \eqref{eq:xy-bound} gives
$D>0$. Applying \eqref{eq:Caro-Wei-consequence} to $G[A]$ and $G[B]$, and
then using the arithmetic--geometric mean inequality, we get
\[
 \begin{aligned}
 L
 &\geq \frac{a^2}{2x}+\frac{(n-a)^2}{2y}-\frac n2\geq \frac{a(n-a)}{\sqrt{xy}}-\frac n2\geq \frac{n^2/4-(a-n/2)^2}{\sqrt{7D}}-\frac n2.
 \end{aligned}
\]
By Lemma~4.4, $(a-n/2)^2\leq6\delta\leq6D$. This gives the first inequality.
The second follows from $n^2/4-6D\geq n^2/8$.
\end{proof}

The reverse estimate is obtained by repacking internal edges into
edge-disjoint triangles.

\begin{lemma}
Let $A\mathbin{\dot\cup}B$ be the partition in Lemma~4.4. If
$D\leq n^2/10^4$ then
\[
 \max\{e(G[A]),e(G[B])\}\leq36\delta.
\]
\end{lemma}

\begin{proof}
Choose $X\in\{A,B\}$ such that
\[
 m=e(G[X])=\max\{e(G[A]),e(G[B])\},
\]
and put $Y=V(G)\setminus X$, $x=|X|$, and $y=|Y|$. Also put
$\mu=\overline e_G(X,Y)$. By Lemma~4.4,
\[
 \mu\leq6\delta,\qquad
 |x-y|\leq2\sqrt{6\delta},\qquad
 x\geq \frac n2-\sqrt{6D}.
\]
If $m=0$, there is nothing to prove. Otherwise, by Vizing's theorem, $G[X]$
has a proper edge-coloring with $q\leq x$ colours, each color class being a
matching. If $q\leq y$, retain every color class. If $q>y$, retain the $y$
largest color classes. Let $\ell$ be the number of edges in the discarded
classes. In the latter case, averaging gives
\[
 \ell\leq\frac{q-y}{q}m\leq\frac{x-y}{x}m.
\]
Thus, in either case,
\begin{equation}
 \ell\leq\frac{|x-y|}{x}m
 \leq\frac{2\sqrt{6D}}{n/2-\sqrt{6D}}m
 \leq\frac{2\sqrt6}{50-\sqrt6}m
 <\frac{m}{9}.
 \label{eq:discarded-edges}
\end{equation}

Assign the retained colour classes injectively to vertices of $Y$. If an edge
$uv$ has a retained colour assigned to $z\in Y$, call $uvz$ a candidate
triangle. It is valid if both $uz$ and $vz$ are edges of $G$. The valid
candidate triangles are pairwise edge-disjoint. Indeed, their edges in $X$ are
distinct, while a crossing edge $wz$ can occur only in the colour class
assigned to $z$, and a matching contains at most one edge incident with $w$.

At most $\mu$ candidates are invalid. To see this, assign to each invalid
candidate one of its missing crossing edges. A fixed missing pair $wz$ can be
assigned to at most one candidate, because the retained colour assigned to
$z$ is unique and its colour class is a matching. If $T$ is the number of
valid candidates, then
\begin{equation}
 T\geq m-\ell-\mu.
 \label{eq:valid-triangles}
\end{equation}

Use the $T$ valid triangles, use two-vertex cliques for all remaining edges
within $X$ and across $(X,Y)$, and use the clique partition of $G[Y]$ supplied
by Lemma~4.4. Since $e_G(X,Y)\leq xy\leq n^2/4$, this gives
\[
 \begin{aligned}
 p
 &\leq T+(m-T)+\bigl(e_G(X,Y)-2T\bigr)+5\delta \\
 &\leq \frac{n^2}{4}-m+2\ell+2\mu+5\delta,
 \end{aligned}
\]
where the second inequality follows from \eqref{eq:valid-triangles}. Since
$p=n^2/4-\delta$, $\mu\leq6\delta$, and \eqref{eq:discarded-edges} holds,
\[
 m\leq2\ell+18\delta<\frac{2m}{9}+18\delta.
\]
Hence $m<162\delta/7<36\delta$.
\end{proof}

We can now compare the two estimates for $L$.

\begin{proof}[Proof of Proposition~4.1]
Recall that
\[
 D=\frac{n^2}{4}-\bigl(\cpn(G)-\ccn(G)\bigr).
\]
If $D>n^2/10^4$, then $D>10^{-4}n^{4/3}$. We may therefore assume that
$D\leq n^2/10^4$. By Lemma~4.6,
\[
 L=e(G[A])+e(G[B])\leq72\delta\leq72D.
\]
Lemma~4.5 now yields
\[
 \frac{n^2}{8\sqrt{7D}}-\frac n2\leq72D,
\]
and hence
\begin{equation}
 \frac{n^2}{8\sqrt7}\leq72D^{3/2}+\frac n2\sqrt D.
 \label{eq:D-main}
\end{equation}

Suppose first that $D\geq n/144$. Since $n/2\leq72D$, the preceding lower bound
for $L$ gives
\[
 \frac{n^2}{8\sqrt{7D}}\leq144D.
\]
It follows that
\begin{equation}
 D\geq(1152\sqrt7)^{-2/3}n^{4/3}.
 \label{eq:D-lower}
\end{equation}

Suppose next that $D<n/144$. The right-hand side of \eqref{eq:D-main} is then
less than $n^{3/2}/12$, whereas its left-hand side is
$n^2/(8\sqrt7)$. This is impossible for all sufficiently large $n$.
Therefore \eqref{eq:D-lower} holds whenever $D\leq n^2/10^4$ and $n$ is
sufficiently large. Taking
\[
 c_0=\min\left\{10^{-4},(1152\sqrt7)^{-2/3}\right\}
\]
completes the proof.
\end{proof}

The case of odd order follows by adding an isolated vertex.

\begin{corollary}
There is an absolute constant $c>0$ such that every graph $G$ on $n$ vertices
satisfies
\[
 \cpn(G)-\ccn(G)
 \leq\left\lfloor\frac{n^2}{4}\right\rfloor-cn^{4/3}
\]
whenever $n$ is sufficiently large.
\end{corollary}

\begin{proof}
If $n$ is even, the result follows directly from Proposition~4.1. Suppose that
$n$ is odd, and add one isolated vertex to $G$. Neither clique parameter
changes. Applying Proposition~4.1 to the resulting graph gives
\[
 \cpn(G)-\ccn(G)
 \leq\frac{(n+1)^2}{4}-c_0(n+1)^{4/3}.
\]
Since
\[
 \frac{(n+1)^2}{4}-\left\lfloor\frac{n^2}{4}\right\rfloor
 =\frac{n+1}{2},
\]
the linear term is at most $c_0(n+1)^{4/3}/2$ for all sufficiently large $n$.
Thus one may take $c=c_0/2$.
\end{proof}

\begin{proof}[Proof of Theorem~1.2]
The lower and upper bounds follow from Corollaries~3.2 and~4.7, respectively.
\end{proof}

\end{document}